\documentclass[a4paper, 12pt,oneside,reqno]{amsart}
\usepackage[a4paper]{geometry}
\usepackage{amssymb,amsmath,amsthm}
\usepackage[T2A]{fontenc}

\usepackage{xcolor}

\usepackage{bbm}
\usepackage{graphicx}

\newtheorem{theorem}{Theorem}

\newtheorem{lemma}{Lemma}
\newtheorem{corollary}{Corollary}

\theoremstyle{definition}

\newtheorem{definition}{Definition}

\def\Var{\mathrm{Var}}

\def\As{\mathrm{As}}

\def\Com{\mathrm{Com}}

\def\Perm{\mathrm{Perm}}

\def\Nov{\mathrm{Nov}}

\def\wt{\mathop{\fam 0 wt}\nolimits}

\let\<\langle
\let\>\rangle

\numberwithin{equation}{section}

\title[Basis of the free noncommutative Novikov algebra]{Basis of the free noncommutative Novikov algebra}

\author{B. K. Sartayev}

\address{Suleyman Demirel University, Kaskelen, Kazakhstan and Sobolev Institute of Mathematics, Novosibirsk, Russia}

\email{baurjai@gmail.com}

\begin{document}

\begin{abstract}
As it is known, the defining identities of a free Novikov algebra can be obtained from a commutative algebra with a derivation. In this paper, we consider a class of algebras obtained from the class of associative algebras with a derivation that generalizes Novikov algebras. Such objects are called noncommutative Novikov algebras. We construct a monomial basis for a free noncommutative Novikov algebra.
\end{abstract}

\maketitle

\section{Introduction}

It is well-known that every associative Lyndon-Shirshov word corresponds to a unique nonassociative Lyndon-Shirshov word by the arrangement of parentheses, and the set of such words forms the basis of the free Lie algebra, see \cite{Shirshov1958}. The analog of that result for the group appeared in \cite{Lyndon1958}.

In the same way, we consider a free associative algebra in the alphabet $X$ with derivation and denote it by $\As\<X\>^{(d)}$. As an analog for the associative Lyndon-Shirshov words, we consider a subset of basis monomials of $\As\<X\>^{(d)}$ of weight $-1$, where the weight function is defined as follows:
\[
\begin{gathered}
\wt(x)=-1,\quad x\in X; \\
\wt(d(u)) = \wt(u)+1; \\
\wt(uv)=\wt(u)+\wt(v).
\end{gathered}
\]
For every monomial of $\As\<X\>^{(d)}$ of the weight $-1$, we define the rule of the arrangements of parentheses and the product operations $\succ$, $\prec$. In that way, we obtain the set of nonassociative monomials with operations $\succ$ and $\prec$. Such a set is the basis of a free noncommutative Novikov algebra, which is defined by the following identities:
\begin{gather}
    x\succ (y\prec z) = (x\succ y)\prec z,\label{eq:LIdent1} \\
    (x\prec y)\succ z - x\succ (y\succ z) = x\prec (y\succ z) - (x\prec y)\prec z. \label{eq:LIdent2}
\end{gather}
The identities \eqref{eq:LIdent1},  \eqref{eq:LIdent2} first appeared
in \cite{Loday2010}. Also, this algebra is derived in \cite{KSO2019}.

The main motivation for this work is as follows: The defining identities of the variety of Novikov algebras come from commutative algebra with derivation under the operation
$$a\circ b=ab'\;\textrm{or}\; a\circ b=a'b.$$
This depends on the definition of the Novikov algebra. We denote by $\Nov$ the variety of Novikov algebras. The variety of noncommutative Novikov algebras is a generalization of the variety of Novikov algebras in the case $a\succ b=b\prec a$ (or $\succ^{op}=\prec$). In this case, the identity (\ref{eq:LIdent1}) becomes the right-commutative (left-commutative) identity and the identity (\ref{eq:LIdent2}) becomes the left-symmetric (right-symmetric) identity. 

Let's take a look at some well-known works on Novikov algebras. The defining identities of the variety of Novikov algebras first appeared in \cite{GD79}. The usage of Novikov algebras for the classification of linear Poisson brackets of hydrodynamic type was given in \cite{BN85}. Results on the classification of simple Novikov algebras were presented in \cite{Xu1996,Xu2001,Zelm}. Although some structural problems of Novikov algebras were studied in works such as \cite{Osborn92,Osborn92-1,Osborn94}, further research is still being conducted on other structural aspects of these algebras. For example, \cite{Pan2022} studied prime and semiprime Novikov algebras. The Specht property for Novikov algebras was established in \cite{DotIsmUmi2023}.
One of the main results for the free Novikov algebras is the monomial basis of $\Nov\<X\>$, which can be found in \cite{DzhLofwall}. Recently, it was shown that every Novikov algebra can be embedded into some commutative algebra with derivation, as described in \cite{BCZ2017} and \cite{PK-BS2021}. The analog of this result for noncommutative Novikov algebras is given in \cite{erlagol2021}: every noncommutative Novikov algebra can be embedded into some associative algebra with derivation. Using the results on the embedding, we obtain the following diagram:

\begin{picture}(30,80)
\put(215,57){$\hookrightarrow$}
\put(160,44){\tiny{$\succ^{op}=\prec$}
\normalsize\rotatebox[origin=c]{270}{$\hookrightarrow$}}
\put(215,32){$\hookrightarrow$}
\put(240,44){{\normalsize\rotatebox[origin=c]{270}{$\hookrightarrow$}}\tiny{$\;\;\cdot^{op}=\cdot$}}
\put(178,31){$\textrm{N-}\Nov$}
\put(183,57){$\Nov$}
\put(233,31){$\As^{(d)}$}
\put(233,57){$\Com^{(d)}$}
\end{picture}
\vspace*{-\baselineskip}

\noindent where $\textrm{N-}\Nov$ and $\Com^{(d)}$ stand for the variety of noncommutative Novikov algebras and variety of commutative algebras with derivation, respectively.

There was an attempt to find the Gröbner-Shirshov basis for the operad N-$\Nov$ in the sense \cite{Bremner-Dotsenko}. However, during the calculation, new non-trivial compositions appeared up to degree 9, and it is expected that this list will only grow. Therefore, we found it necessary to construct a monomial basis of the free $\textrm{N-}\Nov\<X\>$ algebra.

\section{Noncommutative Novikov operad}
In this section, we give some interesting results related to the noncommutative Novikov operad.

The operad $\Nov$ has its own distinguished role in the combinatorics of derivations on non-associative algebras \cite{KSO2019}: given a binary operad $\Var$, the Manin white product $\Nov\circ \Var$ is the operad governing the class of derived $\Var $-algebras. The latter are obtained from $\Var $-algebras with a derivation $d$ relative to the new operations 
\[
a\prec b = a d(b),\quad a\succ b = d(a)b
\]
(for each of binary products in $\Var$, if there are more than one operation).
For that reason the free noncommutative Novikov algebra can be embedded into free associative algebra with derivation, i.e., consider $\Var=\As$.
Another interesting example is $\Var=\Perm$ and one obtains
$\Nov\circ\Perm$ is di-$\Nov$. For more details, see \cite{PK-BS-perm}.

As shown in \cite{erlagol2021}, the subspace $\As_{-1}\<X\>^{(d)}$ spanned 
by all monomials of weight $-1$ in  
$\As\<X\>^{(d)}$ algebra is isomorphic to 
the free noncommutative Novikov algebra $\textrm{N-}\Nov\langle X\rangle$
under the mapping
\[
\tau:\textrm{N-}\Nov\<X\> \rightarrow \As\<X\>^{(d)} 
\]
which is defined by $\tau(a\prec b) = a d(b)$ and $\tau(a\succ b) = d(a)b$.
To obtain the dimension of operad N-$\Nov$ we use the result of \cite{Dzh-nonKoszul}. That is
\[
\dim(\Nov(n))=\binom{2n-2}{n-1}.
\]
Since all monomials of algebra $\Com\<X\>^{(d)}$ of weight~$-1$ are spanning elements of algebra $\Nov\<X\>$ under the mapping $\tau$, we obtain the dimension of the operad $\textrm{N-}\Nov$:
\[
\dim(\textrm{N-}\Nov(n))=n!\binom{2n-2}{n-1}.
\]

Let us consider Koszul dual operad of N-$\Nov$ which is denoted by $\textrm{N-}\Nov^!$. The defining identities of the operad $\textrm{N-}\Nov^!$ are
\begin{gather*}
x_1\dashv(x_2\dashv x_3)=0, \\
(x_1\vdash x_2)\vdash x_3=0,
\end{gather*}
\begin{gather*}
x_1\vdash(x_2\dashv x_3)=(x_1\vdash x_2)\dashv x_3, \\
x_1\dashv(x_2\vdash x_3)=x_1\vdash(x_2\vdash x_3)=(x_1\dashv x_2)\dashv x_3=(x_1\dashv x_2)\vdash x_3.
\end{gather*}
From degree $4$, the dimension of operad $\textrm{N-}\Nov^!$ is $0$ which gives that the nilpotency index of $\textrm{N-}\Nov^!$ is $4$. 

\section{Basis of the free noncommutative Novikov algebra}

In this section, we state the main result of the paper.

Note that the operad $\textrm{N-}\Nov $ is nonsymmetric: the order of variables 
in all defining identities is fixed. Hence (see \cite{Loday2001}), it is enough to find a monomial basis of the 1-generated free 
algebra $F = \textrm{N-}\Nov\langle x\rangle $. Each monomial in $F$ represents 
a binary tree with vertices labeled by the operation symbols $\prec$, $\succ $ (basic tree). In order to get a monomial basis 
of $\textrm{N-}\Nov\langle X\rangle $ 
for an arbitrary set $X$ of generators, 
it is enough to mark the leaves of basic trees by the elements 
of~$X$ in all possible ways.

For every $u = x^{(n_1)}\dots x^{(n_k)}$, $\wt (u)=-1$, 
define a {\em standard bracketing} denoted $[u]$ or $[n_1,\dots , n_k]$.
The latter is a planar binary tree with $k$ leaves
whose vertices are labeled by $\prec $ and $\succ $.
Thus, such a tree is a monomial in free algebra with 
two operations generated by a single variable~$x$.

\begin{definition}\label{basiswords}
Let $u$ be a word in 
$\As\langle x\rangle^{(d)}$
such that $|u|=k+1$, $\wt(u)=-1$. 
If $u = x^{p} x^{(k)} x^{k-p}$ then 
\[
[u] = 
\underset{p}{\underbrace{
x\prec \big (x\prec \big(x\prec \dots \prec \big(x}}\prec 
( \dots((x\succ
\underset{k-p}{\underbrace{
 x)\succ x)\succ \dots \succ x)} }
 \big )\dots \big)\big);
\]
If $u$ contains more than two derivatives then 
there should exist 
a proper subword $v$ of length $>1$ such that $\wt(v)=-1$. 
Choose the shortest leftmost such subword in $u$ of length $>1$, so that 
$u = u_1vu_2$. Assume $i$ is the starting position of the chosen subword $v$ in $u$. Then find the standard bracketings $[w]$ and $[v]$ for the shorter words $w = u_1xu_2$ and $v$ by induction. Finally, define a
$[u]$ by substitution of $[v]$ to the $i$th position in $[w]$. A \textit{normal form} of $u$ is $[u]$. 
\end{definition}

For example, let 
\[
u = x x' x'' x' x'' x x .
\]
The shortest proper leftmost subword of weight $-1$ is 
$v=xx'$: $u = v x'' x' x'' x x$ (single letter $x$ is not the desired 
subword as it is of length 1). Note that 
$[xx'] = x\prec x$ and proceed to 
\[
w = x x'' x' x'' x x .
\]
Here the shortest leftmost subword of weight $-1$ is 
$v_1 = x''xx$: $w = xx''x'v_1$. 
For $w_1 = xx''x'x$, we have
$w_1=xx''v_2$, and $w_2=xx''x$, where $v_2=x'v_1$. By definition, we obtain
\[
[w_2] = x\prec (x\succ x).
\]
Proceed to $v_2 = x'(x''xx)$ and substitute the result into $[w_2]$: 
\[
[v_2] = x\succ ((x\succ x)\succ x ), \quad 
[w] = x\prec (x\succ (x\succ ((x\succ x)\succ x ))).
\]
Finally, 
\[
[u] = (x\prec x)\prec (x\succ (x\succ ((x\succ x)\succ x ))).
\]
 
For each monomial $u\in\As\<X\>^{(d)}$ with a weight $-1$ correspond unique monomial $[u]$. Denote by $U$ the set of monomials $\As\<X\>^{(d)}$ of weight $-1$ and by $[U]$ corresponding set to $U$.

Let us give all monomials of $[U]$ of degree $3$. The monomials of degree $3$ are
$$U_3=\{x''xx,xx''x,xxx'',x'x'x,x'xx',xx'x'\},$$
where $U_3\subset U$ and $U_n$ is a set of words of degree $n$.
The corresponding set to $U_3$ is 
\begin{multline}\label{length3}
[U_3]=\{(x\succ x)\succ x,x\prec(x\succ x),x\prec(x\prec x),x\succ(x\succ x),\\
(x\succ x)\prec x,(x\prec x)\prec x\}
\end{multline}

\begin{theorem}
The set $[U]$ is the basis of the one-generated free noncommutative Novikov algebra.
\end{theorem}

\section{Proof of the main result}
In this section, we prove the main result of the paper.
\begin{proof}
It is enough to show that for every 
weight-homogeneous differential polynomial
$u\in \As\<x\>^{(d)}$ of $\wt(u)=-1$ there exists $[v_1],\ldots,[v_r]\in [U]$ such that 
\begin{equation}\label{equality}
u=\sum_{i=1}^{r} \lambda_i \tau([v_i]), \quad \lambda_i\in \Bbbk .
\end{equation}
Indeed, assume \eqref{equality} holds and 
$w\in \textrm{N-}\Nov\langle x\rangle $
is a homogeneous element of degree $n$.
Then $u=\tau(w)$ is a weight-homogeneous differential polynomial 
which can be presented by \eqref{equality}. Since $\tau $ is injective, we obtain 
\[
w=\sum_{i=1}^{r} \lambda_i [v_{i}], \quad [v_i]\in [U_n].
\]
Hence, the set $[U_n]$ is linearly complete in the space $F_n$
of degree $n$ elements in $ \textrm{N-}\Nov\langle x\rangle$.
Thus this is a basis since $|U_n|=|[U_n]|=\dim F_n$.

The statement \eqref{equality} is enough to prove for monomials only.
Let us apply induction on the degree $k$ of $u$. 
The base of the induction is provided by the obvious cases 
with $k=1,2$ and, for $k=3$, by the set \eqref{length3}. 
Suppose \eqref{equality} holds for all words of degree less than~$k$.

To prove the main result, let us state the following definition:

\begin{definition}\label{rho}
Let $$[u]=[u_1]\prec(\cdots\prec([u_{t-1}]\prec([u_{t}]\prec((\cdots(x\succ [u_{t+1}])\succ\cdots)\succ [u_{t+l}])))\cdots)$$
and
$$u=x^{(n_1)}x^{(n_2)}\cdots x^{(n_p)}\cdots x^{(n_{k-1})}x^{(n_{k})}$$ 
be the words of the sets $[U]$ and $U$, respectively, 
such that $n_p=t+l$.
We define a function $\rho:u\rightarrow \mathbb{Z_+}$ as follows:
$$\rho(u)=p,$$
i.e., this is the position of a generator $x$ with some non-zero differentiation degree from the left side of $u$, given the form $[u_1]\cdots [u_{t}]x^{(t+l)}[u_{t+1}]\cdots[u_{t+l}]$ obtained by rewriting $u$ by Definition \ref{basiswords}.
The generator $x^{(n_p)}$ is called a $head$ of $u$ and we denote it by $h(u)$.
\end{definition}

Let us define an order on generators of $\As\<x\>^{(d)}$ as follows:
$$x^{(a)}>x^{(b)}\;\;\; \textrm{if}\;\;\; a>b.$$
For arbitrary monomials $u$ and $v$ from the set $U$, we define an order inductively as follows: 
$$u>v\;\;\; \textrm{if}\;\;\;    \rho(u)<\rho(v),$$
or
$$\rho(u)=\rho(v)\;\;\;\textrm{and}\;\;\; h(u)>h(v).$$
For the case $\rho(u)=\rho(v)$ and $h(u)=h(v)$, we have
$$[u]=[u_1]\prec(\cdots\prec([u_{t-1}]\prec([u_{t}]\prec((\cdots(x\succ [u_{t+1}])\succ\cdots)\succ [u_{t+l}])))\cdots)$$
and
$$[v]=[v_1]\prec(\cdots\prec([v_{t-1}]\prec([v_{t}]\prec((\cdots(x\succ [v_{t+1}])\succ\cdots)\succ [v_{t+l}])))\cdots),$$
and compare $u_i$ and $v_i$ lexicographically as given above i.e. $u>v$ if $u_i=v_i$ for $i=1,\ldots,p$ and $u_{p+1}>v_{p+1}$.

\begin{lemma}\label{lemma1}
If $\tau([u])=u+\sum_{i}\lambda_i v_i$ and $\rho(u)=p$, then one of the following cases holds:
\begin{itemize}
    \item $\rho(v_i)> p,$
    \item $\rho(v_i)=p$ and $h(u)>h(v_i),$
    \item $\rho(v_i)=p$ and $h(u)=h(v_i).$
\end{itemize}   
\end{lemma}
\begin{proof}
If $u=x^{p-1}x^{(l)}x^{l-p+1}$ then
\begin{equation}\label{basiccase}
  \tau([u])=x(\cdots(x(x((\cdots(x'x)'\cdots)' x)')')'\cdots)'.  
\end{equation}
The condition $\rho(v_i)\geq p$ means that the head of $v_i$ either stays at the same position or moves to the right side relative to the head of $u$.

Let us first consider the subword of $u$ that stays to the left of $h(u)$. The weight of that subword is $-p+1$, because the number of generators to the left of $h(u)$ is $x^{p-1}$. Calculating $\tau([u])$ using Leibniz rule, the maximal weight of the left subword of $v_i$ relative to the head of $u$ is $-1$. By Definition $\ref{basiswords}$, the head of $v_i$ cannot be in that left subword, and we obtain $\rho(v_i)\geq p$. The general case can be proved similarly.

The equalities $\rho(u)=\rho(v_i)$ and $h(u)=h(v_i)$ hold in the case
$$[u]=[u_1]\prec(\cdots\prec([u_{t-1}]\prec([u_{t}]\prec((\cdots(x\succ [u_{t+1}])\succ\cdots)\succ [u_{t+l}])))\cdots)$$
and $v_i$ is an arbitrary monomial of the sum
$$\tau([u_1])\cdots\tau([u_{t-1}])\tau([u_{t}]) x^{(t+l)} \tau([u_{t+1}])\cdots \tau([u_{t+l}]).$$
\end{proof}

\begin{lemma}\label{order}
If $\tau([u])=u+\sum_{i}\lambda_i v_i$, then $u> v_i$.
\end{lemma}
\begin{proof}
By Lemma \ref{lemma1} and the given order, it suffices to consider the case where $v_i$ is an arbitrary monomial of the sum 
$$\tau([u_1])\cdots\tau([u_{t-1}])\tau([u_{t}]) x^{(t+l)} \tau([u_{t+1}])\cdots \tau([u_{t+l}]).$$
Similarly, using the given order, we compare $u_1$ with all monomials of $\tau([u_1])$, and then repeat the same procedure for $u_2,\ldots,u_{t+l}$ with all monomials of $\tau([u_2]),\ldots,\tau([u_{t+l}])$, respectively. This leads us to the case (\ref{basiccase}).
\end{proof}

For given order $>$, the list of monomials with derivation degree one on each generator is
\begin{equation}\label{basicmonomials}
\underset{p}{\underbrace{x'\cdots x'}}\; x>\underset{p-1}{\underbrace{x'\cdots x'}}\; x \; x'>\cdots>
x \underset{p}{\underbrace{x'\cdots x'}}>x'\; x \underset{p-1}{\underbrace{x'\cdots x'}}
.
\end{equation}
Notice that for all monomials of the form (\ref{basicmonomials}) the equality (\ref{equality}) holds, i.e.,
for $w-\underset{t}{\underbrace{x'\cdots x'}}\; x \; \underset{l}{\underbrace{x'\cdots x'}}\;$, i.e., we have
$$w=\tau((\cdots( (\underset{t}{\underbrace{x\succ(\cdots\succ(x}}\succ x)\cdots)) \underset{l}{\underbrace{\prec x)\prec\cdots)\prec x}}).$$

By Lemma \ref{order}, for any $u\in U$ we have
$$u=\tau([u])-\sum_{i}\lambda_i v_i,$$
where $u>v_i$, and we call it a $rewriting\; rule$ of monomial $u$. For $v_i$ we consider 4 cases:

Case 1: If $\rho(u)=1$ then
$$[u]=(\cdots(x\succ [u_{1}])\succ\cdots)\succ [u_{t}]$$
and
\begin{equation*}
u=x^{(t)}x^{(m_1)}x^{(m_2)}\cdots x^{(m_{r})}.    
\end{equation*}
By Lemma \ref{order},
$$u=\tau([u])-\sum_i \lambda_i u_i,$$
where $u>u_i$. For $u_i$, it holds that either it is a monomial in the sum $x^{(t)}\tau([u_1])\tau([u_2])\cdots \tau([u_t])$, or $h(u)>h(u_i)$.
For the first case, applying the rewriting rule on monomials $u_1,\ldots,u_t$ several times, finally, we obtain monomials of the form (\ref{basicmonomials}) and we denote them by $u_1',\ldots,u_r'$, respectively. We obtain
$$u'=x^{(t)}u_1'\cdots u_t'$$
and
$$u'=\tau((\cdots(x\succ [u_{1}'])\succ\cdots)\succ [u_{t}'])-
\sum_{m_1,\ldots,m_{t}}\sum_{i=1}^{r-1} x^{(i)}x^{(m_1)}\cdots x^{(m_{r})}.$$
The remained monomials correspond to the case $h(u)>h(u_i)$.

For the case $h(u)>h(u_i)$, using the rewriting rule several times, we obtain $h(u_i)=1$. For this case, we have 
$$u_i=x'x^{(m_1)}x^{(m_2)}\cdots x^{(m_{r})}$$
By Definition \ref{basiswords} and inductive hypothesis on length of monomial,
$$x^{(m_1)}x^{(m_2)}\cdots x^{(m_{r})}$$
can be written as a sum of $[v_1],\ldots,[v_r]\in[U]$, and
$$u_i=\tau(x\succ(\sum_{i=1}^r \lambda_i[v_i])),$$
where $x\succ[v_i]\in [U]$.

Case 2: $\rho(u)=\rho(v_i)\neq 1$ and $h(u)=h(v_i)$.
If 
\begin{equation}\label{[u]}
[u]=[u_1]\prec(\cdots\prec([u_{t-1}]\prec([u_{t}]\prec((\cdots(x\succ [u_{t+1}])\succ\cdots)\succ [u_{t+l}])))\cdots) 
\end{equation}
and
\begin{equation}\label{u}
u=x^{(n_1)}x^{(n_2)}\cdots x^{(n_p)}\cdots x^{(n_{k-1})}x^{(n_{k})},
\end{equation}
then using rewriting rule on monomials $u_1,\ldots,u_{t+l}$ several times, by Lemma \ref{order}, we finally obtain monomials of the form (\ref{basicmonomials}), which we denote by $u_1',\ldots,u_{t+l}'$. We get the minimal monomial for case 2, which is 
$$u'=u_1'\cdots u_t' x^{(t+l)} u_{t+1}'\cdots u_{t+l}',$$
and applying the rewriting rule to it, we obtain 
\begin{multline*}
u'=\tau([u_1']\prec(\cdots\prec([u_{t}']\prec((\cdots(x\succ [u_{t+1}'])\succ\cdots)\succ [u_{t+l}'])))\cdots))- \\
\sum_{m_1,\ldots,m_{p+l}}\sum_{i=1}^{t+l-1} x^{(m_1)}\cdots x^{(m_{p-1})} x^{(i)} x^{(m_{p+1})}\cdots x^{(m_{p+l})}.
\end{multline*}
The monomial $[u_1']\prec(\cdots\prec([u_{t}']\prec((\cdots(x\succ [u_{t+1}'])\succ\cdots)\succ [u_{t+l}'])))\cdots)\in [U]$ and the remained monomials of the sum correspond to case 3.

Case 3: $\rho(u)=\rho(v_i)\neq 1$ and $h(u)>h(v_i).$ Let $u\in U$. If the derivation degree of the head $v_i$ decreases, then by Lemma \ref{lemma1}, the head of $v_i$ moves to the right side ultimately, i.e., we obtain $\rho(u)<\rho(v_i)$. This corresponds to case 4.

Case 4: $\rho(u)\neq 1$ and $\rho(u)<\rho(v_i)$.
Let $u\in U$, and by using the rewriting rule on $u$ and the cases 2-3, the head of $v_i$ always moves to the right side. If $u$ has a form (\ref{[u]}) and (\ref{u}), then the head comes to the rightmost generator ultimately. Finally, we get a certain monomial $w$ of the following form:
\begin{equation*}
[w]=[w_1]\prec(\cdots\prec([w_{r-1}]\prec([w_{r}]\prec x))\cdots)    
\end{equation*}
and
\begin{equation*}
w=x^{(m_1)}x^{(m_2)}\cdots x^{(m_{t})}x^{(r)}.    
\end{equation*}
By Lemma \ref{order},
$$w=\tau([w])-\sum_i \lambda_i w_i,$$
where $w>w_i$. It is possible in the case $w_i$ is a monomial of the sum 
$$\tau([w_1])\tau([w_2])\cdots \tau([w_r])x^{(r)}$$ or $h(w)>h(w_i)$.
For the first case, applying the rewriting rule on monomials $w_1,\ldots,w_r$ several times, finally, we obtain monomials of the form (\ref{basicmonomials}) and we denote them by $w_1',\ldots,w_r'$, respectively. We obtain
$$w'=w_1'\cdots w_r'x^{(r)}$$
and
$$w'=\tau([w_1']\prec(\cdots\prec([w_{r-1}']\prec([w_{r}']\prec x))\cdots))-
\sum_{m_1,\ldots,m_{t}}\sum_{i=1}^{r-1} x^{(m_1)}\cdots x^{(m_{t})} x^{(i)}.
$$
The remained monomials correspond to the case $h(w)>h(w_i)$.

For the case $h(w)>h(w_i)$, using the rewriting rule several times, we obtain $h(w_i)=1$. For this case, we have 
$$w_i=x^{(m_1)}x^{(m_2)}\cdots x^{(m_{t})}x'$$
By Definition \ref{basiswords} and inductive hypothesis on length of monomial,
$$x^{(m_1)}x^{(m_2)}\cdots x^{(m_{t})}$$
can be written as a sum of $[v_1],\ldots,[v_r]\in[U]$, and
$$w_i=\tau((\sum_{i=1}^r \lambda_i[v_i])\prec x),$$
where $[v_i]\prec x\in [U]$.

\end{proof}

\begin{corollary}
To obtain a basis of $\textrm{N-}\Nov\<X\>$ algebra, we place all possible permutations of the alphabet $X$ in monomials $[U]$ of the corresponding length.
\end{corollary}

\begin{center} ACKNOWLEDGMENTS   \end{center}
This research was funded by the Science Committee of the Ministry of Science and Higher Education of the Republic of Kazakhstan (Grant No. AP14871710).
The author is grateful to Professor P.S. Kolesnikov for his comments and advice, thanks to which the article has acquired such a full-fledged look.

\end{document}